\documentclass[twoside,leqno,10pt]{amsart}
\usepackage{amsfonts}
\usepackage{amsmath}
\usepackage{amscd}
\usepackage{amssymb}
\usepackage{amsthm}
\usepackage{latexsym}
\usepackage{color}
\newtheorem{theorem}{Theorem}
\newtheorem{Proposition}{Proposition}

\numberwithin{equation}{section}

\title{Small Pythagorean triples modulo prime powers}
\author{Stephan Baier \and Anup Haldar} 

\address{Stephan Baier,
Ramakrishna Mission Vivekananda Educational and Research Institute, Department of Mathematics, G. T. Road, PO Belur Math, Howrah, West Bengal 711202, India}
\email{stephanbaier2017@gmail.com}
\address{Anup Haldar,
Ramakrishna Mission Vivekananda Educational and Research Institute, Department of Mathematics, G. T. Road, PO Belur Math, Howrah, West Bengal 711202, India}
\email{anuphaldar1996@gmail.com}
\subjclass[2020]{11L40,11T23,11K36} 
\keywords{quadratic congruences, Poisson summation, evaluation of complete exponential sums}

\begin{document}
\begin{abstract}
Let $p>5$ be a fixed prime. We obtain an asymptotic formula related to small solutions of quadratic congruences of the form $x_1^2+x_2^2\equiv x_3^2\bmod{p^n}$ where $\max\{|x_1|,|x_2|,|x_3|\}\le p^{\nu n}$ with $\nu>1/2$. 
\end{abstract}
\maketitle
\tableofcontents

\section{Introduction and main results} 
Let $Q(x_1,...,x_n)$ be a quadratic form with integer coefficients. The question of detecting small solutions of congruences of the form
$$
Q(x_1,...,x_n)\equiv 0 \bmod{q}
$$
has received a lot of attention (see, in particular, \cite{Hea1}, \cite{Hea2} and \cite{Hea3}). Of particular interest is the case when 
$$
Q(x_1,...,x_n)=\alpha_1x_1^2+\cdots +\alpha_kx_k^2
$$ 
is a diagonal form and $q=p^n$ is a prime power. This was considered by Hakimi in \cite{Hak}, with emphasis on quadratic forms with a large number $k$ of variables. Here we want to focus on the case $k=3$. In this case, a result by Schinzel, Schlickewei and Schmidt \cite{SSS} for general moduli $q$ implies that there is a non-zero solution $(x_1,x_2,x_3)\in \mathbb{Z}^3$ such that $\max\{|x_1|,|x_2|,|x_3|\}=O(q^{2/3})$, where the $O$-constant is absolute.  For $q$ square-free, the exponent $2/3$ was improved to $3/5$ by Heath-Brown \cite{Hea3}. 
A result by Cochrane \cite{Coc} for general moduli $q$ implies that for any {\it fixed} form $Q(x)$, there is a non-zero solution with $\max\{|x_1|,|x_2|,|x_3|\}=O(q^{1/2})$, where the $O$-constant may depend on the form. In the present paper, we are interested in {\it asymptotic formulas} for the number of non-zero solutions in boxes $\max\{|x_1|,|x_2|,|x_3|\}\le N$ with as small as possible $N$ if $q$ is a prime power. In contrast, the above-mentioned works deal with the {\it existence} of solutions.  We shall consider a smoothed version of this problem with the special choice $\alpha_1=1$, $\alpha_2=1$, $\alpha_3=-1$, i.e., the congruences in question are of the form
\begin{equation} \label{pythpn}
x_1^2+x_2^2\equiv x_3^2\bmod{p^n}.
\end{equation}
We further restrict ourselves to $x_i$'s which are coprime to $p$, which automatically excludes the trivial solution $(0,0,0)$. Thus, we investigate the distribution of Pythagorean triples in $\left((\mathbb{Z}/p^n\mathbb{Z})^{\ast}\right)^3$.  

Small solutions of \eqref{pythpn} arise immediately from Pythagorean triples in $\mathbb{Z}^3$, provided that $p>5$. (If $p=2,3,5$, we have $x_1^2+x_2^2-x_3^2\not\equiv 0\bmod{p}$ if $(x_1x_2x_3,p)=1$.) It is known that the number of Pythagorean triples $(x_1,x_2,x_3)\in \mathbb{Z}^3$ satisfying $x_1^2+x_2^2=x_3^2$ such that $|x_3|\le N$ is $\sim cN\log N$ with $c=8/\pi$ (see Proposition \ref{Pytha} below). It should not be difficult to modify this into an asymptotic of the form $\sim c_pN\log N$ with $c_p$ depending on $p$ if one includes the restriction $(x_1x_2x_3,p)=1$. If $N< \sqrt{q/2}$ with $q=p^n$, then any solution $(x_1,x_2,x_3)$ of the congruence \eqref{pythpn} is in fact a Pythagorean triple. Hence, in this case, one expects  an asymptotic of the form $\sim c_pN\log N$ for the number of solutions satisfying $(x_1x_2x_3,p)=1$ and $\max\{|x_1|,|x_2|,|x_3|\}\le N$ of the said congruence. In contrast, for much larger $N$, the expected number of solutions should be $\sim d_pN^3/q$ for a suitable constant $d_p>0$. In particular, one may expect this to hold for $N\ge q^{1/2+\varepsilon}$. Hence, there should be a transition between two different asymptotic formulas near the point $N=q^{1/2}$. Indeed, we will work out an asymptotic of the said form $\sim d_pN^3/q$ for $N\ge q^{\nu}$ with $\nu>1/2$. Our precise result is as follows. 
 
\begin{theorem} \label{mainresult}
Let $\varepsilon>0$ and $\nu>1/2$ be fixed, $\Phi:\mathbb{R}\rightarrow \mathbb{R}_{\ge 0}$ be a Schwartz class function and $p>5$ be a prime. Then as $n\rightarrow \infty$, we have the asymptotic formula
\begin{equation} \label{main}
\sum\limits_{\substack{(x_1,x_2,x_3)\in \mathbb{Z}^3\\ (x_1x_2x_3,p)=1\\ x_1^2+x_2^2-x_3^2 \equiv 0 \bmod{p^n}}} \Phi\left(\frac{x_1}{N}\right)
\Phi\left(\frac{x_2}{N}\right)\Phi\left(\frac{x_3}{N}\right)\sim
\hat{\Phi}(0)^3\cdot \frac{(p-s(p))(p-1)}{p^2}\cdot \frac{N^3}{p^{n}}
\end{equation}
if $N\ge  p^{n\nu}$, where 
\begin{equation} \label{sdef}
s(p):=\begin{cases}
3 & \mbox{ if } p\equiv 3\bmod{4},\\
5 & \mbox{ if } p\equiv 1\bmod{4}.
\end{cases}
\end{equation}
\end{theorem}

Our original motivation was to derive a quantitative $p$-adic version of the Oppenheim conjecture for particular ternary quadratic forms. Indeed, our problem can be re-interpreted as detecting small integers $x_1,x_2,x_3$ of $p$-adic norm 1 such that the $p$-adic norm of $Q(x_1,x_2,x_3)=x_1^2+x_2^2-x_3^2$ is also small.

Key ingredients in the method are a parametrization of $\mathbb{Q}_p$-rational points $(z_1,z_2)$ on the circle
$$
z_1^2+z_2^2=1,
$$
repeated use of Poisson summation and an explicit evaluation of complete exponential sums with rational functions to prime power moduli due to Chochrane \cite{CoZ}. This transforms the problem into a dual problem which amounts to counting ordinary Pythagorean triples in $\mathbb{Z}^3$.  

It should be possible to obtain a non-smoothed version of Theorem \ref{mainresult} along similar lines. However, the technical details become then more complicated.  Moreover, the authors believe that a generalization of the method to arbitrary fixed diagonal forms $Q(x_1,x_2,x_3)$ and general moduli $q$ is possible and that a similar result as in Theorem \ref{mainresult} holds for boxes that are not centered at 0, leading to equidistribution of triples. This may be subject to future research. 

Another interesting question is how the number of solutions behaves in the transition range near the point $N=q^{1/2}$. Here a more natural approach is to count representations of $f_m(x_3)=mq+x_3^2$ as sums of two squares, where $|m|$ is small. \\ \\
{\bf Acknowledgements.} The authors would like to thank the Ramakrishna Mission Vivekananda Educational and Research Institute for providing excellent working conditions. The second-named author would like to thank CSIR, Govt. of India for financial support in the form of a Junior Research Fellowship. 

\section{Preliminaries}
The following preliminaries will be needed in the course of this paper. We will  use the notation 
$$
e_q(z):=e\left(\frac{z}{q}\right)=e^{2\pi i z/q}
$$
for $q\in \mathbb{N}$ and denote by $G_q$ the quadratic Gauss sum
$$
G_q=\sum\limits_{x=1}^q e_q(x^2).
$$
We recall that if $q$ is odd, then 
$$
G_q=\sum\limits_{y=1}^q \left(\frac{y}{q}\right),
$$
where $\left(\frac{y}{q}\right)$ is the Jacobi symbol. We further recall that in this case, $|G_q|=\sqrt{q}$.

\begin{Proposition}[Parametrization of points on a circle] \label{para}
Let $K$ be a field. Then all $K$-rational points on the circle
\begin{equation*} 
z_1^2+z_2^2=1
\end{equation*}
are parametrized in the form
$$
(z_1,z_2)=\left(\frac{1-t^2}{1+t^2},\frac{2t}{1+t^2}\right),
$$
where $t\in K$ with $t^2\not=-1$. The map $m : \{t\in K : t^2\not=-1\} \longrightarrow \{(z_1,z_2)\in K^2 : z_1^2+z_2^2=1\}$, defined by 
$$
m(t)=\left(\frac{1-t^2}{1+t^2},\frac{2t}{1+t^2}\right),
$$
is bijective. 
\end{Proposition}

\begin{proof}
See \cite{notes}.
\end{proof}

\begin{Proposition}[Poisson summation formula] \label{Poisson} Let $\Phi : \mathbb{R}\rightarrow \mathbb{R}$ be a Schwartz class function, $\hat\Phi$ its Fourier transform. Then
$$
\sum\limits_{n\in \mathbb{Z}} \Phi(n)=\sum\limits_{n\in \mathbb{Z}} \hat\Phi(n).
$$
\end{Proposition}

\begin{proof}
See \cite[section 3]{notes}. 
\end{proof}

\begin{Proposition}[Evaluation of exponential sums with rational functions] \label{Expsums}
Let $p>2$ be a prime, $n\ge 2$ be a natural number and $f=F_1/F_2$ be a rational function where $F_1,F_2\in \mathbb{Z}[x]$. For a polynomial $G$ over $\mathbb{Z}$, let $\mbox{ord}_p(G)$ be the largest power of $p$ dividing all of the coefficients of $G$, and for a rational function $g=G_1/G_2$ with $G_1$ and $G_2$ polynomials over $\mathbb{Z}$, let $\mbox{ord}_p(g) := \mbox{ord}_p(G_1)-\mbox{ord}_p(G_2)$. Set
$$
r:=\mbox{ord}_p(f')
$$
and 
$$
S_{\alpha}(f;p^n):=\sum\limits_{\substack{x=1\\ x\equiv \alpha \bmod{p}}}^{p^n} e_{p^n}(f(x)), 
$$
where $\alpha\in \mathbb{Z}$.
Then we have the following if $r\le n-2$ and $(F_2(\alpha),p)=1$.\medskip\\
(i) If $p^{-r}f'(\alpha)\not\equiv 0\bmod{p}$, then $S_{\alpha}(f,p^n) = 0$.\medskip\\
(ii) If $\alpha$ is a root of the congruence $p^{-r}f'(x)\equiv 0\bmod{p}$ of multiplicity one, then
$$
S_{\alpha}(f;p^n) =\begin{cases} e\left(f(\alpha^{\ast})\right)p^{(n+r)/2} & \mbox{ if } n-r \mbox{ is even,}\\
e\left(f(\alpha^{\ast})\right)p^{(n+r)/2}\left(\frac{A(\alpha)}{p}\right)\cdot \frac{G_p}{\sqrt{p}} & \mbox{ if } n-r \mbox{ is odd,}
\end{cases}
$$
where $\alpha^{\ast}$ is the unique lifting of $\alpha$ to a solution of the congruence $p^{-r}f'(x) \equiv  0 \bmod p^{[(n-r+1)/2]}$, and 
$$
A(\alpha):=2p^{-r}f''(\alpha^{\ast}).
$$
\end{Proposition}

\begin{proof} This is \cite[Theorem 3.1(iii)]{CoZ}.
\end{proof}

\begin{Proposition} \label{Pytha}
We have  
$$
\sharp\{(x_1,x_2,x_3)\in \mathbb{Z}^3  : x_1^2+x_2^2=x_3^2,\ |x_3|\le N\}\sim \frac{8}{\pi}N\log N
$$
as $N\rightarrow \infty$. 
\end{Proposition}

\begin{proof} This is found in \cite{NoR} and originally due to Sierpinski \cite{Sie}. \end{proof}

\section{Proof of the main result}
\subsection{Parametrization of points} The congruence above resembles the equation 
\begin{equation} \label{circle}
x_1^2+x_2^2-x_3^2=0
\end{equation} 
of a circle in homogeneous coordinates. Let us first see that every solution $(x_1,x_2,x_3)\in \mathbb{Z}^3$ with $(x_1x_2x_3,p)=1$ of the congruence 
\begin{equation} \label{congru1}
x_1^2+x_2^2-x_3^2\equiv 0 \bmod{p^n}
\end{equation}
above comes from a solution of \eqref{circle} in the $p$-adic integers. To this end, we need to use a Hensel-type argument. Let $(x_1,x_2,x_3)$ be such a solution. We would like to lift it to a solution $(\tilde{x}_1,\tilde{x}_2,\tilde{x}_3)$ of the congruence
\begin{equation} \label{congru2}
\tilde{x}_1^2+\tilde{x}_2^2-\tilde{x}_3^2\equiv 0\bmod{p^{n+1}}.
\end{equation} 
So we consider $\tilde{x}_i=x_i+k_ip^n$ with $k_i=0,...,p-1$, $i=1,2,3$ and satisfying
$$
(x_1+k_1p^n)^2+(x_2+k_2p^n)^2-(x_3+k_3p^n)^2\equiv 0 \bmod{p^{n+1}}.
$$ 
Expanding the squares and using $2n\ge n+1$, this is equivalent to 
$$
x_1^2+x_2^2-x_3^2+2k_1p^nx_1+2k_2p^nx_2-2k_3p^nx_3\equiv 0 \bmod{p^{n+1}},
$$
which in turn is equivalent to
$$
\frac{x_1^2+x_2^2-x_3^2}{p^n}+2x_1k_1+2x_2k_2-2x_3k_3\equiv 0\bmod{p}.
$$
This linear congruence in $k_1,k_2,k_ 3$ has exactly $p^2$ solutions. In particular, a solution $(x_1,x_2,x_3)$ of \eqref{congru1} lifts to a solution $(\tilde{x}_1,\tilde{x}_2,\tilde{x}_3)$ of \eqref{congru2}.  So indeed, every solution of \eqref{congru1} arises from a solution of \eqref{circle} in $\mathbb{Z}_p$. Now we parametrize these solutions. This works in a similar way as for Pythagorean triples. First, by Proposition \ref{para}, the $\mathbb{Q}_p$-rational points $(z_1,z_2)$ on the circle
\begin{equation} \label{conic}
z_1^2+z_2^2=1
\end{equation}
are parametrized as 
$$
(z_1,z_2)=\left(\frac{1-t^2}{1+t^2},\frac{2t}{1+t^2}\right),
$$
where $t\in \mathbb{Q}_p$ with $1+t^2\not=0$. Now if $(x_1,x_2,x_3)$ is a solution of \eqref{circle} in the $p$-adic integers, where $x_3\not=0$, then
$$
\left(\frac{x_1}{x_3},\frac{x_2}{x_3}\right)
$$
is a point on the circle in \eqref{conic}. Hence, we have 
$$
\left(\frac{x_1}{x_3},\frac{x_2}{x_3}\right)=\left(\frac{1-t^2}{1+t^2},\frac{2t}{1+t^2}\right).
$$
But we restricted ourselves to triples $(x_1,x_2,x_3)$ with $|x_i|_p=1$, $i=1,2,3$. Hence, we have 
$$
1=\left|\frac{x_1}{x_3}\right|_p=\left|\frac{1-t^2}{1+t^2}\right|_p
$$
and 
$$
1=\left|\frac{x_2}{x_3}\right|_p=\left|\frac{2t}{1+t^2}\right|_p.
$$
If $|t|_p<1$, then
$$
\left|\frac{2t}{1+t^2}\right|_p=|t|_p<1,
$$
and if $|t|_p>1$, then
$$
\left|\frac{2t}{1+t^2}\right|_p=|t|_p^{-1}<1.
$$
Hence $|t|_p=|1+t^2|_p=|1-t^2|_p=1$, and 
$$
x_1=(1-t^2)u, \quad x_2=2tu, \quad x_3=(1+t^2)u, 
$$
where $u$ is a unit in the ring of integers $\mathbb{Z}_p$, i.e. $|u|_p$=1. 

By reducing modulo $p^n$, we deduce that the solutions of the congruence
$$
y_1^2+y_2^2-1\equiv 0 \bmod{p^n}
$$
with $(y_1y_2,p)=1$ are parametrized in the form
\begin{equation} \label{para1}
y_1=\frac{1-t^2}{1+t^2}, \quad y_2=\frac{2t}{1+t^2}, \quad t \bmod{p^n},\ (t(1-t^2)(1+t^2),p)=1.
\end{equation} 
Here $1/(1+t^2)$ stands for a multiplicative inverse of $1+t^2 \bmod{p^n}$. Moreover, the pairs $(y_1,y_2)$ given as in \eqref{para1} are distinct modulo $p^n$ by the following argument: Suppose that  
\begin{equation} \label{t1t21}
\frac{1-t_1^2}{1+t_1^2}\equiv \frac{1-t_2^2}{1+t_2^2} \bmod{p^n}
\end{equation}
and 
\begin{equation} \label{t1t22}
\frac{2t_1}{1+t_1^2}\equiv \frac{2t_2}{1+t_2^2} \bmod{p^n}.
\end{equation}
Then from \eqref{t1t21} it follows upon multiplying both sides with the denominators that 
$$
1-t_1^2t_2^2+t_2^2-t_1^2\equiv 1-t_1^2t_2^2+t_1^2-t_2^2\bmod{p^{n}}
$$
and hence 
$$
t_1^2\equiv t_2^2\bmod{p^n}.
$$
So if $t_1\not\equiv t_2\bmod{p^n}$, then $t_1\equiv -t_2\bmod{p^n}$. However, in this case, \eqref{t1t22} implies $t_1\equiv t_2\equiv 0\bmod{p^n}$ contradicting the assumption that $(t_i,p)=1$ for $i=1,2$. 

\subsection{Double Poisson summation} We start by writing 
\begin{equation*}
\begin{split}
T= & \sum\limits_{\substack{(x_1,x_2,x_3)\in \mathbb{Z}^3\\ (x_1x_2x_3,p)=1\\ x_1^2+x_2^2-x_3^2\equiv 0 \bmod{p^n}}} \Phi\left(\frac{x_1}{N}\right)\Phi\left(\frac{x_2}{N}\right)\Phi\left(\frac{x_3}{N}\right)\\
=& \sum\limits_{(x_3,p)=1} \Phi\left(\frac{x_3}{N}\right)\sum\limits_{\substack{y_1,y_2\bmod{p^n}\\ (y_1y_2,p)=1 \\ y_1^2+y_2^2-1\equiv 0\bmod{p^n}}} \sum\limits_{\substack{x_1\equiv x_3y_1\bmod{p^n}\\ x_2\equiv x_3y_2\bmod{p^n}}} \Phi\left(\frac{x_1}{N}\right)\Phi\left(\frac{x_2}{N}\right).
\end{split}
\end{equation*}
Now we apply Poisson summation, Proposition \ref{Poisson}, after a linear change of variables to the inner double sum over $x_1$ and $x_2$, obtaining
\begin{equation*}
T= \frac{N^2}{p^{2n}}\sum\limits_{(x_3,p)=1} \Phi\left(\frac{x_3}{N}\right)\sum\limits_{(k_1,k_2)\in \mathbb{Z}^2} \hat{\Phi}\left(\frac{k_1N}{p^n}\right)\hat{\Phi}\left(\frac{k_2N}{p^n}\right)\sum\limits_{\substack{y_1,y_2\bmod{p^n}\\ (y_1y_2,p)=1 \\ y_1^2+y_2^2-1\equiv 0\bmod{p^n}}} e_{p^n}\left(k_1x_3y_1+k_2x_3y_2\right).
\end{equation*}
Using our parametrization \eqref{para}, we deduce that
\begin{equation*}
T= \frac{N^2}{p^{2n}}\sum\limits_{(x_3,p)=1} \Phi\left(\frac{x_3}{N}\right)\sum\limits_{(k_1,k_2)\in \mathbb{Z}^2} \hat{\Phi}\left(\frac{k_1N}{p^n}\right)\hat{\Phi}\left(\frac{k_2N}{p^n}\right)\sum\limits_{\substack{t \bmod{p^n}\\
(t(1-t^2)(1+t^2),p)=1 }} e_{p^n}\left(x_3\cdot \frac{k_1(1-t^2)+2k_2t}{1+t^2}\right).
\end{equation*}

We decompose $T$ into 
\begin{equation} \label{divide}
T=T_0+U,
\end{equation}
where $T_0$ is the main term contribution of $(k_1,k_2)=(0,0)$. Hence,
\begin{equation*}
T_0= \hat{\Phi}(0)^2\cdot \frac{N^2}{p^{2n}}\sum\limits_{(x_3,p)=1} \Phi\left(\frac{x_3}{N}\right) \cdot p^{n-1}(p-s(p)),
\end{equation*}
where $s(p)$ is defined as in \eqref{sdef}. (Note that $1+t^2\equiv 0 \bmod{p}$ has two solutions modulo $p$ if $p\equiv 1 \bmod{4}$ and no solution if $p\equiv 3\bmod{4}$.) 
If $N\ge p^{n\varepsilon}$ for any fixed $\varepsilon>0$, then the term $T_0$ can be simplified into
\begin{equation*}
\begin{split}
T_0= & \hat{\Phi}(0)^2 \cdot \frac{p-s(p)}{p}\cdot \frac{N^2}{p^n} \cdot \left(\sum\limits_{x} \Phi\left(\frac{x}{N}\right) -\sum\limits_{x} \Phi\left(\frac{x}{N/p}\right)\right)\\
= & \hat{\Phi}(0)^2 \cdot \frac{p-s(p)}{p}\cdot \frac{N^2}{p^n}\cdot \left(N\cdot \frac{p-1}{p}\cdot \hat\Phi(0)
+\sum\limits_{y\in \mathbb{Z}\setminus\{0\}} \left(N\hat\Phi(Ny)-\frac{N}{p}\cdot \hat\Phi\left(\frac{Ny}{p}\right)\right)\right)\\
= & \hat{\Phi}(0)^3 \cdot \frac{(p-s(p))(p-1)}{p^2}\cdot \frac{N^3}{p^{n}}\cdot\left(1+o(1)\right)
\end{split}
\end{equation*}
as $n\rightarrow\infty$, 
where we again use Poisson summation for the sums over $x$ above and the rapid decay of $\hat\Phi$.   

\subsection{Evaluation of exponential sums} Now we look at the error contribution
\begin{equation} \label{errorcont}
U= \frac{N^2}{p^{2n}}\sum\limits_{(x_3,p)=1} \Phi\left(\frac{x_3}{N}\right)\sum\limits_{(k_1,k_2)\in \mathbb{Z}^2\setminus \{(0,0)\}} \hat{\Phi}\left(\frac{k_1N}{p^n}\right)\hat{\Phi}\left(\frac{k_2N}{p^n}\right) E\left(k_1,k_2,x_3;p^n\right)
\end{equation}
with 
\begin{equation*}
E\left(k_1,k_2,x_3;p^n\right):=\sum\limits_{\substack{t \bmod{p^n}\\
(t(1-t^2)(1+t^2),p)=1 }} e_{p^n}\left(x_3\cdot \frac{k_1(1-t^2)+2k_2t}{1+t^2}\right).
\end{equation*}
Assume that
$$
(k_1,k_2,p^n)=p^r.
$$
The contributions of $r=n-1$ and $r=n$ to the right-hand side of \eqref{errorcont} are $O_{\varepsilon}(1)$ if $N\ge p^{n\varepsilon}$ by the rapid decay of $\hat\Phi$. (Recall that $(k_1,k_2)=(0,0)$ is excluded from the summation.) In the following, we assume that $r\le n-2$ so that Proposition \ref{Expsums} is applicable.

We split the inner-most sum over $t$ into 
\begin{equation*} \label{split}
E\left(k_1,k_2,x_3;p^n\right)=\sum\limits_{\substack{\alpha=1\\ \alpha \not\equiv 0,\pm 1 \bmod{p}\\ 1+\alpha^2\not\equiv 0\bmod{p}}}^p S_{\alpha}\left(f_{k_1,k_2};p^n\right)=\sum\limits_{\substack{\alpha=1\\ \alpha^2 \not\equiv 0,\pm 1 \bmod{p}}}^p S_{\alpha}\left(f_{k_1,k_2};p^n\right),
\end{equation*}
where 
\begin{equation*} \label{Salphadef}
S_{\alpha}\left(f_{k_1,k_2,x_3};p^n\right)=\sum\limits_{\substack{t \bmod{p^n}\\ t\equiv \alpha\bmod{p}}} e_{p^n}\left(f_{k_1,k_2,x_3}(t)\right)
\end{equation*}
with 
\begin{equation*} \label{fdef}
f_{k_1,k_2,x_3}(t):=x_3\cdot \frac{k_1(1-t^2)+2k_2t}{1+t^2}.
\end{equation*}
We calculate that 
\begin{equation*}
\begin{split}
f_{k_1,k_2,x_3}'(t)
= & 2x_3\cdot \frac{k_2(1-t^2)-2k_1t}{(1+t^2)^2}.
\end{split}
\end{equation*}
Set
$$
l_1:=\frac{k_1}{p^r}, \quad l_2:=\frac{k_2}{p^r}.
$$
Then using Proposition \ref{Expsums}, if $\alpha^2+1\not=0$, we have $S_{\alpha}\left(f_{k_1,k_2,x_3};p^n\right)=0$ unless 
\begin{equation} \label{keycong}
2l_1\alpha\equiv l_2(1-\alpha^2) \bmod{p}.
\end{equation}
If $\alpha\not=0,\pm 1$, then it follows that $(l_1l_2,p)=1$ and 
$$
(k_1,p^n)=p^r=(k_2,p^n).
$$
In summary, we have 
\begin{equation*}
U= \frac{N^2}{p^{2n}}\sum\limits_{(x_3,p)=1} \Phi\left(\frac{x_3}{N}\right)\sum\limits_{r=0}^{n-2} \sum\limits_{\substack{\alpha=1\\ \alpha^2 \not\equiv 0,\pm 1 \bmod{p}}}^p \sum\limits_{\substack{(l_1,l_2)\in \mathbb{Z}^2\\ (l_1l_2,p)=1\\ 2l_1\alpha\equiv l_2(1-\alpha^2) \bmod{p}}} \hat{\Phi}\left(\frac{l_1N}{p^{n-r}}\right)\hat{\Phi}\left(\frac{l_2N}{p^{n-r}}\right) S_{\alpha}\left(f_{p^rl_1,p^rl_2};p^n\right)+O_{\varepsilon}(1)
\end{equation*}
if $N\ge p^{n\varepsilon}$. 

Let $D:=l_1^2+l_2^2$. The congruence \eqref{keycong} has a double root $\alpha \bmod{p}$ iff $D\equiv 0 \bmod{p}$, and in this case we get $\alpha^2\equiv -1\bmod{p}$ which is excluded from the summation over $\alpha$. Hence, only the case $D\not\equiv 0\bmod{p}$ occurs in which we have no root if $D$ is a quadratic non-residue modulo $p$ and two roots of multiplicity one if $D$ is a quadratic residue modulo $p$. Therefore, we may assume from now on that $D\not\equiv 0 \bmod{p}$ and $D$ is a quadratic residue modulo $p$. Then  using Proposition \ref{Expsums}, if $\alpha$ satisfies \eqref{keycong}, we obtain
$$
S_{\alpha}\left(f_{p^rl_1,p^rl_2},p^n\right)=
\begin{cases}
e_{p^n}\left(f_{p^rl_1,p^rl_2}(\alpha^{\ast})\right)\cdot p^{(n+r)/2} & \mbox{ if } n-r \mbox{ is even,}\\
e_{p^n}\left(f_{p^rl_1,p^rl_2}(\alpha^{\ast})\right)\cdot \left(\frac{A(\alpha)}{p}\right)\cdot \frac{G_p}{\sqrt{p}}\cdot p^{(n+r)/2} &  \mbox{ if } n-r \mbox{ is odd,}
\end{cases}
$$
where $\alpha^{\ast}$ is the unique lifting of $\alpha$ to a root of the congruence
$$
2l_1\alpha^{\ast}\equiv l_2(1-(\alpha^{\ast})^2) \bmod{p^{n-r}}
$$
and 
$$
A(\alpha)=\frac{2 f_{p^rl_1,p^rl_2}''(\alpha^{\ast})}{p^r}.
$$
We calculate that 
$$
\alpha^{\ast}\equiv \frac{-l_1\pm\sqrt{D}}{l_2} \bmod{p^{n-r}},
$$
where $\sqrt{D}$ denotes one of the two roots of the congruence
$$
x^2\equiv D\bmod{p^{n-r}}.
$$
A short calculation gives
$$
e_{p^n}\left(f_{p^rl_1,p^rl_2}(\alpha^{\ast})\right)=e_{p^{n-r}}\left(\pm x_3\sqrt{D}\right).
$$
Further, we calculate the second derivative of $f_{p^rl_1,p^rl_2}$ to be
$$
f_{p^rl_1,p^rl_2}''(t)=6x_3\cdot \frac{-l_1-3l_2t+3l_1t^2+l_2t^3}{(1+t^2)^3}\cdot p^{r}.
$$
A short calculation gives 
$$
A(\alpha)=\frac{2x_3l_2^4}{\sqrt{D}(l_1\mp\sqrt{D})^2}
$$
and hence 
$$
\left(\frac{A(\alpha)}{p}\right)=\left(\frac{2x_3\sqrt{D}}{p}\right). 
$$

It is easy to check that the cases $\alpha^2\equiv 0,\pm 1\bmod{p}$ cannot occur if $\alpha$ is a root of multiplicity one of the congruence \eqref{keycong} with $(l_1l_2,p)=1$.  So altogether, we obtain
\begin{equation} \label{Uaftereva}
\begin{split}
U= & \frac{N^2}{p^{3n/2}}\sum\limits_{r=0}^{n-2} p^{r/2} \sum\limits_{\substack{(l_1l_2,p)=1\\ D=\Box\bmod{p}}} \hat{\Phi}\left(\frac{l_1N}{p^{n-r}}\right)\hat{\Phi}\left(\frac{l_2N}{p^{n-r}}\right)\times \\ & \sum\limits_{(x_3,p)=1} \Phi\left(\frac{x_3}{N}\right)\cdot C_{n-r}(x_3,D)\cdot 
\left(e_{p^{n-r}}\left(x_3\sqrt{D}\right)+ e_{p^{n-r}}\left(-x_3\sqrt{D}\right)\right)+O_{\varepsilon}(1),
\end{split}
\end{equation}
where $D=\Box\bmod{p}$ means that $D$ is a quadratic residue modulo $p$ and  
$$
C_{n-r}(x_3,D):=\begin{cases} 1 & \mbox{ if } n-r \mbox{ is even,}\\ 
\left(\frac{2x_3\sqrt{D}}{p}\right)\cdot \frac{G_p}{\sqrt{p}} & \mbox { if } n-r
\mbox{ is odd.} \end{cases}
$$
If $(x_3D,p)=1$, then
$$
C_{n-r}(x_3,D)= \frac{G_{p^{n-r}}}{p^{(n-r)/2}}\cdot 
\left(\frac{2x_3\sqrt{D}}{p^{n-r}}\right)
$$
in each of the two cases above. Therefore, $U$ can be more compactly written as 
\begin{equation} \label{compact}
\begin{split}
U= & \frac{N^2}{p^{3n/2}}\sum\limits_{r=0}^{n-2} p^{r/2} \cdot \frac{G_{p^{n-r}}}{p^{(n-r)/2}} \cdot \sum\limits_{\substack{D=1\\ D\equiv \Box \bmod{p}}}^{\infty} F_{n-r}(D) \times\\
 & \sum\limits_{x_3\in \mathbb{Z}} \Phi\left(\frac{x_3}{N}\right)\cdot \left(\frac{x_3}{p^{n-r}}\right) \cdot 
\left(e_{p^{n-r}}\left(x_3\sqrt{D}\right)+ e_{p^{n-r}}\left(-x_3\sqrt{D}\right)\right)+O_{\varepsilon}(1),
\end{split}
\end{equation}
where 
\begin{equation} \label{FD}
F_{n-r}(D):= \left(\frac{2\sqrt{D}}{p^{n-r}}\right)\cdot \sum\limits_{\substack{(l_1l_2,p)=1\\ l_1^2+l_2^2=D}} \hat{\Phi}\left(\frac{l_1N}{p^{n-r}}\right)\hat{\Phi}\left(\frac{l_2N}{p^{n-r}}\right).
\end{equation}

\subsection{Single Poisson summation and final count} 
Now we split the sum over $x_3$ in \eqref{compact} into subsums over residue classes modulo $p$ and perform Poisson summation, getting
\begin{equation*}
\begin{split}
\sum\limits_{x_3\in \mathbb{N}} \Phi\left(\frac{x_3}{N}\right)\cdot \left(\frac{x_3}{p^{n-r}}\right) \cdot e_{p^{n-r}}\left(\pm x_3\sqrt{D}\right) = 
& \sum\limits_{u=1}^{p} \left(\frac{u}{p^{n-r}}\right)
\sum\limits_{x_3\equiv u\bmod{p}} \Phi\left(\frac{x_3}{N}\right)\cdot
e_{p^{n-r}}\left(\pm x_3\sqrt{D}\right)\\
= & \frac{N}{p} \sum\limits_{v\in \mathbb{Z}} \left(\sum\limits_{u=1}^{p} \left(\frac{u}{p^{n-r}}\right) \cdot e_{p}(uv)\right) 
\cdot \hat\Phi\left(\frac{N}{p}\left(\frac{\pm \sqrt{D}}{p^{n-r-1}}-v\right)\right).
\end{split}
\end{equation*}
Using the rapid decay of $\hat\Phi$, the above is $O(N)$ if $||\sqrt{D}/p^{n-r-1}||\le p^{1+n\varepsilon}N^{-1}$ and negligible otherwise, provided $n$ is large enough. We may constraint $\sqrt{D}$ to the range $0\le \sqrt{D}\le p^{n-r}$ and then write $\sqrt{D}=wp^{n-r-1}+l_3$, where $w=0,...,p-1$ and $|l_3|\le L_r$ with
$$
L_r:=p^{n-r+n\varepsilon}N^{-1}.
$$  
Moreover, the summations over $l_1$ and $l_2$ in \eqref{FD} can be cut off at $|l_1|,|l_2|\le L_r$ at the cost of a negligible error if $n$ is large enough. It follows that
 \begin{equation*} 
\begin{split}
U\ll & \frac{N^3}{p^{3n/2}}\sum\limits_{r=0}^{n-2} p^{r/2} \sum\limits_{w=0}^{p-1} \sum\limits_{\substack{(l_1,l_2,l_3)\in \mathbb{Z}^3\setminus\{(0,0,0)\}\\ |l_1|,|l_2|,|l_3|\le L_r\\ l_1^2+l_2^2\equiv (wp^{n-r-1}+l_3)^2\bmod{p^{n-r}}}} 1 +O_{\varepsilon}(1).
\end{split}
\end{equation*}
The congruence above implies 
$$
l_1^2+l_2^2\equiv l_3^2\bmod{p^{n-r-1}}.
$$
Hence,
 \begin{equation*} 
\begin{split}
U\ll & \frac{N^3}{p^{3n/2}}\sum\limits_{r=0}^{n-2} p^{r/2} \sum\limits_{\substack{(l_1,l_2,l_3)\in \mathbb{Z}^3\setminus\{(0,0,0)\}\\ |l_1|,|l_2|,|l_3|\le L_r\\ l_1^2+l_2^2\equiv l_3^2\bmod{p^{n-r-1}}}} 1 +O_{\varepsilon}(1).
\end{split}
\end{equation*}
Now if $L_r< \sqrt{p^{n-r-1}/2}$, then the congruence above can be replaced by the equation $l_1^2+l_2^2=l_3^2$, i.e., $(l_1,l_2,l_3)$ is an ordinary Pythagorean triple. Certainly, this is the case if $N\ge p^{n/2+2n\varepsilon}$ and $n$ is large enough. Hence, in this case, we have
\begin{equation*} 
\begin{split}
U\ll & \frac{N^3}{p^{3n/2}}\sum\limits_{r=0}^{n-2} p^{r/2} \sum\limits_{\substack{(l_1,l_2,l_3)\in \mathbb{Z}^3\setminus\{(0,0,0)\}\\ |l_1|,|l_2|,|l_3|\le L_r\\ l_1^2+l_2^2=l_3^2}} 1 +O_{\varepsilon}(1).
\end{split}
\end{equation*}

Now we apply Proposition \ref{Pytha} to bound $U$ by 
\begin{equation*} 
\begin{split}
U\ll & \frac{N^3}{p^{3n/2}}\sum\limits_{r=0}^{n-2} p^{r/2} L_r^{1+\varepsilon} \ll \frac{N^2}{p^{n/2}}\cdot p^{9n\varepsilon}.
\end{split}
\end{equation*}
This needs to be compared to the main term which is of size
$$
T_0\asymp \frac{N^3}{p^n}.
$$
If $N\ge p^{(1/2+10\varepsilon)n}$, then $U=o(T_0)$. This completes the proof of Theorem \ref{mainresult}.

\end{document}